 \documentclass[final,3p,times]{elsarticle}

 \usepackage{epsfig}

\usepackage{amssymb}
 \usepackage{amsthm}
\usepackage{amsmath,amssymb,amsopn,amsfonts,mathrsfs,amsbsy,amscd}

\usepackage{longtable}
\usepackage{caption}
\usepackage{multirow}

 \usepackage{lineno}

\newcommand{\prs}{\langle\;,\;\rangle}
\newcommand{\too}{\longrightarrow}

\newcommand{\nil}{\mathrm{Nil} }
\newcommand{\sol}{\mathrm{Sol} }
\newcommand{\Ric}{\mathrm{Ric} }

\newcommand{\esp}{\quad\mbox{and}\quad}

\newcommand{\G}{{\mathfrak{g}}}
\newcommand{\g}{{\mathfrak{g}}}
\newcommand{\F}{{\cal F}}
\newcommand{\h}{{\mathfrak{h} }}
\newcommand{\n}{{\mathfrak{n} }}
\newcommand{\ad}{{\mathrm{ad}}}
\newcommand{\tr}{{\mathrm{tr}}}

\newcommand{\B}{{\cal B}}

\newcommand{\na}{\nabla}
\newcommand{\wi}{\widetilde}
\newcommand{\al}{\alpha}

\newcommand{\Ga}{\Gamma}
\def\Euc{\mathsf{Euc}}
\newcommand{\e}{\epsilon}

\newcommand{\ch}{\check{e}}
\newcommand{\la}{\lambda}

\font\bb=msbm10

\def\B{\hbox{\bb B}}
\def\R{\hbox{\bb R}}

\def\N{\hbox{\bb N}}

\def\C{\hbox{\bb C}}
\newtheorem{Def}{Definition}[section]
\newtheorem{theorem}{Theorem}[section]
\newtheorem{pr}{Proposition}[section]
\newtheorem{Le}{Lemma}[section]
\newtheorem{co}{Corollary}[section]
\newtheorem{remarks}{Remarks}[section]

\newtheorem{exem}{Example}
\newtheorem{remark}{Remark}

\def\SO{{\sf{SO}}}
 
\def\SL{{\sf{SL}}}

\def\Iso{{\sf{Iso}} }

\def\ad{{\sf{ad}}}

\usepackage{hyperref}
\hypersetup{
	colorlinks=true,
	linkcolor=blue,
	filecolor=magenta,      
	urlcolor=cyan,
}

\begin{document}

\begin{frontmatter}


 
 
 \title{Kundt Three Dimensional Left Invariant Spacetimes}

 \author[label1]{ Aissa Meliani}
 \author[label2]{  Mohamed Boucetta }
 \author[label3]{ Abdelghani Zeghib }
 \address[label1]{Laboratory of Nonlinear Partial Differential Equations and History of Mathematics, 
 	B.P. 92, 16006-ENS Kouba, Algeria.
 	\\e-mail: Aissa.Meliani.math@gmail.com  
 }
 \address[label2]{Universit\'e Cadi-Ayyad\\
  Facult\'e des sciences et techniques\\
  BP 549 Marrakesh Morocco.\\e-mail: m.boucetta@uca.ma
  }
  
  \address[label3]{UMPA, CNRS, ENS de Lyon, France.
  	\\e-mail: abdelghani.zeghib@ens-lyon.fr  
  }



\begin{abstract} Kundt spacetimes are of great importance to General Relativity. We show that a Kundt spacetime is a Lorentz manifold  with a non-singular isotropic geodesic  vector field  
having its orthogonal distribution  integrable and determining  a totally geodesic foliation. We give the local structure of Kundt spacetimes and some properties of left invariant Kundt structures on Lie groups. Finally, we classify all  left invariant Kundt structures on three dimensional simply connected unimodular Lie groups.

\end{abstract}



\end{frontmatter}








\section{Introduction}\label{section1}

Kundt spacetimes are of great importance to General Relativity, as well as alternative gravity theories. 
To begin with, let us say that Kundt spacetimes 
constitute a natural generalization of
 pp-wave spacetimes. 
  Roughly speaking, a Kundt spacetime is defined by the fact that it supports a vector 
  field  all of its scalar invariants vanish, without being Killing (see for instance \cite[Chapter 6]{SKMHH}).
  One of our motivations here is to provide a coordinate-free treatment of Kundt spacetimes, which is 
  hard to find in 
  the General Relativity literature. More precisely, let us define 
a Kundt spacetime  as a Lorentz  manifold $(M,g)$ having the following property (see  for instance \cite{Hervic}):
\begin{enumerate}\item[$(\mathrm{K}1)$] There exists
    a non-singular vector field $V$ on $M$ such that  
  \begin{equation}\label{eq1}  g(V,V)=g(\na_VV,\na_VV)=\tr(A^V)= g(B^V,B^V)=g(\mathrm{C}^V,\mathrm{C}^V)=0, \end{equation} where $\na$ is the Levi-Civita connection, $A^V: TM\too TM$ denotes the $(1,1)$ tensor field given by $A^V(X)=\na_XV$, $B$ its symmetric part, $C$ its skew-symmetric part
  and if $F:TM\too TM$ is a $(1,1)$-tensor field then $g(F,F)=\tr(F\circ F^*)$ where $F^*$ is its adjoint with respect to $g$.
  \end{enumerate}
  It turns out (see Proposition \ref{equiv}) that this property is equivalent to:
  \begin{enumerate}\item[$(\mathrm{K}2)$] There exists 
  	a non-singular vector field $V$ on $M$ and a differential 1-form $\al$ such that
  	\begin{equation}\label{eq2} g(V,V)=0, \na_XV=\al(X) V\esp \na_VV=0\end{equation}
  	for any vector field $X$ orthogonal to $V$.
  	
  	\end{enumerate}

\bigskip

A fundamental observation for us was (see Proposition \ref{equiv} and more details in Sections \ref{section2},  
\ref{Section Kundt groups} and 
\ref{Section Kundt Spacetimes vs Kundt Groups})  that  Kundt  properties $(\mathrm{K}1)$ or $(\mathrm{K}2)$ imply the following property:
\begin{enumerate}\item[$(\mathrm{LK})$] There exists on $M$ a codimension one  totally geodesic   foliation which is degenerate with respect to $g$. More precisely, there exists a codimension one foliation $\F$ such that each leaf $L$ of $\F$ is a lightlike (totally) geodesic hypersurface
(that is 
$T^\perp L \subset TL$ and any  geodesic $\gamma: [a, b] \to M$, $a < 0 <b$,  somewhere tangent to $L$,  is locally contained in $L$: if $\gamma^\prime (0) \in T_{\gamma(0)} L$ then there exists $\epsilon >0$ such that $\gamma ( [- \epsilon, + \epsilon]) \subset  L$).

	\end{enumerate}


	Actually, up to assuming the direction field $T^\perp \F$ orientable (which is always possible up to passing to a double covering), 
	the property $(\mathrm{LK})$ is equivalent to:

	\begin{enumerate}\item[$(\mathrm{LKbis})$] there exists  
		a non-singular vector field $V$ on $M$ and a differential 1-form $\al$ such that
		\begin{equation}\label{eq3} g(V,V)=0, \na_XV=\al(X) V\end{equation}
		for any vector field $X$ orthogonal to $V$.
		
	\end{enumerate}
	
  We will refer to a Lorentz manifold satisfying $(\mathrm{LK})$   as a {\it local Kundt spacetime}. Indeed, 
  $(\mathrm{LK})$	 implies $(\mathrm{K}1)$ locally near any point in $M$. In other words,  a spacetime 
	admitting a codimension one lightlike geodesic foliation is locally Kundt. 
	
 


\bigskip

Another major motivation to study  Kundt Lorentz manifolds lies in their relation with 
CSI-spaces,   those having all of their scalar curvature invariants are constant.  The simplest 
one is the scalar curvature $\mathsf{Scal}_g$, but one can also consider eigenvalues of the Ricci operator
$\Ric_g$ or the $g$-norm of the Riemann tensor $\mathsf{Rm}_g$.  All those are scalar curvature invariants of order 1. Higher order ones
are obtained by considering covariant derivatives 
of $\mathsf{Rm}_g$.
So CSI means, in particular,  that all  these quantities are   constant functions on $M$.

Locally homogeneous spaces are CSI, and the 
existence  of CSI spaces that are 
not   locally-homogeneous 
spaces   is a non-Riemannian phenomena which makes  a one major difference between the positive and the non-definite cases in pseudo-Riemannian structures. A typical example is given by 
the  (conformally flat) plane wave metric
$g = dx^2 +dy^2 - 2dvdu - 2 f(u)(x^2 +y^2)du^2$, which is VSI (i.e. has vanishing scalar invariants) for any $f$, but locally homogeneous for 
only few $f$'s.  This example is Kundt, in fact $V = \frac{\partial}{\partial v}$ is a parallel vector field.

It is believed that a CSI Lorentz space, if it is not locally homogeneous, must be of Kundt type!  This conjecture has been  proved in some 
cases, e.g. in the lower dimensions 3 and 4, see for instance \cite{CHP1,  CHP2}.
In another direction, there is a notion of $\mathcal I$-degenerate metrics, meaning that they have non-trivial (i.e. non-isometric) 
deformations keeping all the scalar curvature invariant functions  the same (non-depending of the deformation parameter, but maybe 
depending on the point of $M$).  Those are believed to be Kundt too.

Not all locally homogeneous spacetimes are Kundt, neither all Kundt spacetimes are locally homogeneous, but 
it is worthwhile to consider locally homogeneous Kundt spaces as a special class of both the homogeneous and the Kundt categories! Our project is to study  Kundt structures on  three dimensional Lie groups $G$ endowed with a left invariant Lorentzian metric $g$.  It is natural, in this special homogeneous framework, to introduce a stronger Kundt property as follows. 	We call   $(G,g)$  a  {\it Kundt Lie group} (resp. {\it locally Kundt Lie group}) if it admits a non-singular left invariant vector field $V$  satisfying $(\mathrm{K}2)$ (resp. $(\mathrm{LKbis})$). In other words, we assume here compatibility between the
$(\mathrm{K}2)$ or the  $(\mathrm{LKbis})$ property and the algebraic structure of $G$.

	
	
	

\subsection{Results}
	
	One of our principal results, Theorem \ref{main}, states,  essentially,  that a three dimensional Lorentz group which is Kundt as a spacetime, is  in fact a Kundt group.  We also classify, up to isometric isomorphism,  all unimodular three dimensional Kundt groups.


The paper is organized as follows. In Section \ref{section2}, we provide a synthetic (coordinate-free) account on Kundt spacetimes emphasising on their relationship with lightlike geodesic foliations. We introduce Kundt groups and general facts about them in Section \ref{Section Kundt groups}. The proof of Theorem \ref{main} as well as further results are given in Section \ref{Section Kundt Spacetimes vs Kundt Groups}.
Section \ref{section3} contains the classification up to automorphism of Kundt Lorentz groups.





\section{Kundt spacetimes and geodesic foliations}
\label{section2}

Recall from the introduction that a Kundt spacetime is a Lorentz manifold satisfying the property $(\mathrm{K}1)$. Let us prove that this property is equivalent to $(\mathrm{K2})$ and implies $(\mathrm{LK})$. Moreover, $(\mathrm{LK})$ implies $(\mathrm{K1})$ locally near any point of the Lorentz manifold.

\begin{pr}\label{equiv} Let $(M,g)$ be a Lorentz manifold. Consider the following assertions:
	\begin{enumerate}\item[$(i)$] $(M,g)$ is a Kundt spacetime.
		\item[$(ii)$] There exists on $M$ an isotropic non-singular   vector field $V$ and a differential 1-form $\alpha$ such that, for any $X\in\Ga(V^\perp)$,
		\[ \na_XV=\al(X)V\esp \na_VV=0. \]

		\item[$(iii)$] There exists on $M$ a totally geodesic codimension one foliation which is degenerate with respect to $g$. This means that there exists a vector sub-bundle $F\subset TM$ of rank $(\dim M)-1$, where the restriction of $g$ to $F$ is degenerate and, for any $X,Y\in\Ga(F)$, $\na_XY\in\Ga(F)$ where $\na$ is the Levi-Civita connection of $g$.

	\end{enumerate}
	Then $(i)$ and $(ii)$ are equivalent and  both imply $(iii)$. Moreover, $(iii)$ implies that $(ii)$ holds in a neighbourhood of  any point in $M$.

\end{pr}

\begin{proof} $(i)\Longrightarrow(ii).$  Assume  that $(M,g)$ is a Kundt spacetime. This means that there exists a non-singular vector field $V$ satisfying \eqref{eq1}.  Fix a point $p\in M$ and denote by $A$ to the endomorphism given by $Au=\na_uV$ for any $u\in T_pM$. Choose an isotropic vector $U\in T_pM$ such that $g(U,V)=1$ and an orthonormal basis $(e_1,\ldots,e_{n-2})$ of $\left\{U,V_p \right\}^\perp$. Note that
	\[ g(A(V),A(V))=g(V_p,V_p)=g(A(V),V)=0. \] Thus
	the vector subspace $\mathrm{span}\{A(V), V \}$ is totally isotropic, hence its dimension  equals 1 which means  $A(V)=\al_0 V$ for some $\al_0\in\R$. On the other hand, since $g(V,V)=0$, for any $u\in T_pM$, $g(A(u), V)=g(\na_uV,V)=0$ hence $A(T_pM)\subset V^\perp$. This implies that for any $u\in T_pM$, $g(A(u),A(u))\geq0$ and $g(A(u),A(u))=0$ if and only if $A(u)=\al(u)V$.  Now, from \eqref{eq1},
	\[ 0=\tr(A^*A)=2g(A(V),AU)+\sum_{i=1}^{n-2}g(A(e_i),A(e_i))=\sum_{i=1}^{n-2}g(A(e_i),A(e_i)).  \]
	Moreover, we have
	\[ 0=\tr(A)=g(AV,U)+g(AU,V)+\sum_{i=1}^{n-2}g(Ae_i,e_i)=\al_0 \] hence
	$\na_VV=0$. This completes the proof of $(i)\Longrightarrow(ii)$.

	Let us prove  $(ii)\Longrightarrow(i)$. Fix a point $p$ and consider $A$ and $(U,V,e_2,\ldots,e_{n-2})$ as defined above. We have $A(T_pM)\subset \R V$ hence $V^\perp\subset\ker A^*$. With this fact in mind, we get

	\begin{align*} \tr(A)&=g(AU,V)+g(AV,U)+\sum_{i=2}^ng(Ae_i,e_i)=0,\\
	\tr(BB^*)&=2g(BU,BV)+\sum_{i=2}^ng(Be_i,Be_i)=0,\\
	\tr(CC^*)&=2g(CU,CV)+\sum_{i=2}^ng(Ce_i,Ce_i)=0.
	\end{align*}
	This completes the proof of $(i)\Longrightarrow(ii)$.
	
	Let us prove now that $(ii)\Longrightarrow(iii)$.
	Let $F=V^\perp$. We have, for any $X,Y\in\Ga(F)$,
	\[ g(\na_XY, V)=-g(Y,\na_XV)=-\al(X)g(Y,V)=0 \] hence $\na_XY\in\Ga(F)$. This shows that $F$ is integrable and defines a degenerate codimension one totally geodesic foliation.
	
	Now, we prove that if $(iii)$ holds then, for any $p\in M$, there exists a vector field $V$ near $p$ satisfying \eqref{eq1}. So, 
	suppose that there exists an integrable degenerate codimension one   sub-bundle $F\subset TM$ which  defines a totally geodesic foliation. Fix a point $p\in M$ and choose a non-singular vector field $V\in\Ga(F^\perp)\subset\Ga(F)$ near $p$. It is obvious that $V$ is isotropic. Moreover, since $\na_VV\in\Ga(F)$ and $g(\na_VV,V)=0$ then  $\na_VV=\al V$. As above,  denote by $A$ the endomorphism $A_p^V$ and choose a basis $(u,V_p,e_1,\ldots,e_{n-2})$. We have, obviously, that $A(T_pM)\subset F_p$.  Moreover, since $F$ is totally geodesic, for any $X,Y\in\Ga(F)$,
	\[ 0=g(\na_XY,V)=-g(Y,\na_XV) \]which implies that $A(F_p)\subset\R V$. So far, we have shown that locally near $p$, for any $X\in \Ga(F)$,
	\[ AV=\al_0V,\, A(X)=\al(X)V. \]
	To finish  the proof, we look for a vector field $V'=e^fV$ where $f$ is a function such that $V'$ satisfies $(ii)$. This is equivalent to
	 $V(f)=-\al_0$ and such a function exists locally.
	\end{proof}

\bigskip

\subsection{Kundt coordinates}

Another way to compare the Kundt property with the existence of a codimension one lightlike geodesic foliation is given by the following 
fact which asserts  the existence of adapted local coordinates associated to 
lightlike geodesic foliations, where the metric has a special form. The  same adapted coordinates are known to characterize   Kundt spacetimes.

\begin{pr} Let $(M,g)$ be a Lorentz manifold satisfying $(iii)$ of Proposition \ref{equiv}. Then near any point in $M$ there exists a local coordinates system  $(v, u, x =  (x^2, \ldots x^n))$
	where the metric has the form:
	
	$$g =  2dudv+H(v,u,x)du^2+\sum_{i=2}^nW_i (v, u, x) dudx^i
	  +\sum_{i,j} h_{ij} (u, x) dx^i dx^j.$$

	\end{pr}
	
	\begin{remarks} - Observe that the functions $h_{ij}$ do not depend on $v$.
	
	- The foliation in Proposition \ref{equiv} corresponds to the (local) $u$-levels.
	
	- One can also show the converse, that  a  foliation admitting an adapted chart where the metric has such a form 
	is  lightlike geodesic.
	
	\end{remarks}

	\begin{proof} Suppose that there exists a vector sub-bundle $F$ of $TM$ of rank $(\dim M)-1$ such that the restriction of the metric to $F$ is degenerate and the $\Ga(F)$ is stable by the Levi-Civita product.

		Fix a point $p\in M$ and let $\Sigma$  be a local hypersurface  containing $p$ and transversal
		to $F^\perp$.  Then $F_\Sigma=F\cap T\Sigma$ determines a  foliation  on $\Sigma$  hence there exists a  
		 coordinates system
		$({{x^2}}, \ldots,{x^n}, {u}  )$ on $\Sigma$ such that  the  leaves of $F_\Sigma$ are the ${u}$-levels. 
		There is  a  section $T:\Sigma \too (F^\perp)_{|\Sigma}$ such that
		$g( T, 	\frac{\partial}{ \partial {u} })  = 2$.  Choose an injective immersion $\phi:\R^{n-1}\too M$ such that $\phi(\R^{n-1})=\Sigma$. Then there exists $\e>0$ such that the map $\Phi:\R^{n-1}\times(-\e,\e)\too M$ given by $\Phi(t,s)=\exp_{\phi(t)}(s T)$ is a diffeomorphism into its image. Denote by $V$ the image by $\Phi$ of the vector field $\frac{\partial}{\partial s}$. Since $F^\perp$ is totally geodesic, then $V$ is tangent to $F^\perp$  hence satisfies $g(V,V)=0$. By construction, we have $\na_VV=0$ and  according to the proof of Proposition \ref{equiv}, for any $X\in\Ga(F)$, $\na_XV=\al(X)V$.
		
		On the other hand, the vector fields $\frac{\partial}{\partial { u}} , \frac{\partial}{\partial {x^2} }, \ldots, \frac{\partial}{\partial {x^n}   }$ on $\Sigma$ define a family of vector fields on $\R^{n-1}$ hence define a family of vector fields on $\R^{n-1}\times(-\e,\e)$ which commute with $\frac{\partial}{\partial s}$. Let $U,X_2,\ldots,X_n$ be their images by $\Phi$. 
		  We deduce that $V, U, X_2, \ldots, X_n$ are commuting, and  give rise to a local coordinate system $(v, u, x^2, \ldots, x^n)$ on $M$ such that
		\[ V=\frac{\partial}{\partial v},\; U=\frac{\partial}{\partial u}\esp X_i=\frac{\partial}{\partial x_i},\quad i=2,\ldots,n.  \]
		Observe now that, for any vector field $Z$ commuting with $V$, the scalar product $g(V,Z) $ is constant along the $V$-trajectories.
		Indeed, since $g(V,V)=0$ and $[Z,V]=0$ we get
		\[ V.g(V,Z)=g(\na_VV,Z)+g(V,\na_VZ)=g(V,\na_ZV)=0. \]
		We deduce that  for any $i=2,\ldots,n$, $g(V,X_i)$ and $g(U,V)$ are constant along the trajectories of $V$ and since they are constant along $\Sigma$ we get that $g(U,V)=2$ and $g(V,X_i)=0$. Moreover, we have
		\[ V.g(X_i,X_j)=g(\na_VX_i,X_j)+g(X_i,\na_VX_j)=g(\na_{X_i}V,X_j)+
		g(X_i,\na_{X_j}V)=\al(X_i)g(V,X_j)+\al(X_j)g(X_i,V)=0. \]
		This completes the proof.
	\end{proof}

\begin{remark} For a general codimension one  lightlike foliation, not necessarily geodesic, we have similar adapted coordinates, but with 
the functions $h_{ij}$ depending also on $v$. 

\end{remark}

\section{Kundt Groups} \label{Section Kundt groups}

A {\it Lorentz Lie group} is a Lie group $G$ endowed with a left invariant Lorentzian metric $g$. Denote by $\G$ the Lie algebra of $G$ and $\prs=g(e)$. We call $(\G,\prs)$ a Lorentz Lie algebra. The Levi-Civita product is the product $\bullet$ on $\G$ given by
\begin{equation}\label{lc}
2\langle u\bullet v,w\rangle =\langle[u,v],w\rangle+ \langle[w,u],v\rangle+\langle[w,v],u\rangle,\quad u,v,w\in\G.
\end{equation}

A {\it Kundt Lie group} (resp. {\it locally Kundt Lie group}) is a Lorentz Lie group $(G,g)$ having an isotropic left invariant vector field satisfying \eqref{eq2} (resp \eqref{eq3}). 

\begin{pr}\label{kundt} Let $(G,g)$ be a connected Lorentz Lie group. Then the following are equivalent:
	\begin{enumerate}\item[$(i)$] $(G,g)$ is a locally Kundt Lie group.
		
		\item[$(ii)$] There exists  a codimension one subalgebra $\h$ of $\G$ which is degenerate and stable by the Levi-Civita product. 
		
		\end{enumerate}
		
		Moreover, if $(G,g)$ is a locally  Kundt Lie group then its is a Kundt Lie group if and only if, for any generator   $e$  of  $\h^\perp$, $e\bullet e=0$.

	\end{pr}

\begin{proof}  Let us prove that $(i)$ implies $(ii)$.  $(G,g)$ is a locally Kundt Lie group if and only if there exists a left invariant vector field $V$  satisfying \eqref{eq3}. The codimension one subspace $\h=V(e)^\perp$ is degenerate and, for any $u,v\in\h$, denote by $u^l$ and $v^l$ the corresponding left invariant vector fields. Then
	\[ g(\na_{u^l}v^l,V)=u^l.g(v^l,V)-g(v^l,\na_{u^l}V)=-\al(u^l)g(v^l,V)=0 \]
	 hence $\na_{u^l}v^l(e)=u\bullet v\in\h$. This means that $\h$ is stable by the Levi-Civita product, which completes the proof of $(i)\Longrightarrow(ii)$.
	
	Le us show now that $(ii)\Longrightarrow(i)$. Suppose there exists $\h$ a codimension one  degenerate subalgebra of $\G$ which is     stable by the Levi-Civita product and consider $v$ a generator of $\h^\perp$. Denote by $V$ the left invariant vector field associated to $v$. Then, according to the proof of Proposition \ref{equiv}, we have, for any $x\in\h$,
	\[ \na_{x^l}V=\al(x^l)V\esp \na_VV=\al_0 V \]where $\al_0$ is a constant. The last assertion is obvious. \end{proof}
	
	\begin{Def}\label{def} Let $\G$ be a Lie algebra.  A Kundt pair on $\g$ is a pair $(\prs,\h)$ where $\prs$ is a Lorentzian product on $\G$ and $\h$ is a $\prs$-degenerate codimension one subalgebra stable by the Levi-Civita product $\bullet$ given by \eqref{lc} and for any $e\in\h^\perp$, $e\bullet e=0$.
		
		\end{Def}
	
	There is a large class of Lie groups which cannot carry a locally   Kundt group structure.
	
	\begin{Le}\label{compact} Let $\G$ be a semi-simple compact Lie algebra. Then $\G$ cannot have a codimension one subalgebra.
		
		\end{Le}
	
	\begin{proof}
	Suppose that $\G$ has a codimension one Lie subalgebra $\h$. Since $\G$ is compact, it carries  a bi-invariant scalar product $\prs$, i.e., $\ad_x$ is skew-symmetric for any $x\in\G$. 
	For any $x\in\h$, $\ad_x$ leaves $\h$ invariant and since it is skew-symmetric  it leaves $\h^\perp=\R e$ invariant, so $[x,e]=0$ and  $e$ is central which contradicts the fact that $\G$ is semi-simple.
	\end{proof}
	
	\begin{co} Let $G$ be a compact semi-simple Lie group. Then $G$ cannot carry any structure of locally Kundt Lie group.

		\end{co}
	
	The oscillator group named so by Streater in \cite{streater}, as a four-dimensional connected, simply connected
	Lie group, whose Lie algebra (known as the oscillator algebra) coincides with the one generated
	by the differential operators, acting on functions of one variable, associated to the harmonic
	oscillator problem. The oscillator group has been generalized to any even dimension $2n \geq 4$,
	and several aspects of its geometry have been intensively studied, both in differential geometry
	and in mathematical physics (see \cite{bm, calvaruso, duran, gadea, levi, muller}). Oscillator Lie groups have a natural left invariant Kundt structure.
	
	\begin{exem} For $n \in \N^*$ and $\la = (\la_1,\ldots,\la_n) \in \R^n$ with $0 < \la_1 \leq \cdots \leq \la_n$, the $\la$-oscillator group, denoted by $G_\la$, is the Lie group with the underlying manifold $\R^{2n+2} = \R \times \R\times \C^n$ and product $$(t,s,z).(t',s',z') = \Bigl( t+t',s+s'+\frac12 \sum_{j=1}^n \mbox{Im }[\bar{z}_j \exp(it\la_j)z_j'],
		\ldots, z_j+\exp(it\la_j)z'_j, \ldots \Bigr) .$$
		Its Lie algebra  $\G_\la$ is $\R \times \R\times \C^n$ with its canonical basis
		$\B = \left\{e_{-1},e_0,e_j,\ch_j\right\}_{j=1,\ldots,n}$  such that
		\[ e_{-1}=(1,0,0),e_0=(0,1,0),\; e_j=(0,0,(0,\ldots,1,\ldots,0))\esp \ch_j=
		(0,0,(0,\ldots,\imath,\ldots,0)). \]
		and the Lie brackets  are given by
		\begin{equation}\label{bracket}
		[e_{-1},e_i] = \la_i \ch_i, \qquad[e_{-1},\ch_i] = -\la_ie_i, \qquad[e_i,\ch_i] = e_0,
		\end{equation}
		for $i=1,\ldots,n$, the unspecified products are either given by antisymmetry or  zero.  For $x\in\G_\la$, let
		$$x=x_{-1}e_{-1}+x_0e_0+\sum_{i=1}^n(x_ie_i+\check{x}_i\check{e}_i).$$
		The non-degenerate symmetric bilinear form
		\begin{equation}\label{lorentz}\textbf{k}_\la(x,x):=2x_{-1}x_0+\sum_{j=1}^n\frac1{\la_j}(x_j^2+\check{x}_j^2)\end{equation}satisfies
		$$\textbf{k}_\la([x,y],z)+\textbf{k}_\la(y,[x,z])=0,\quad\mbox{for
			any}\quad x,y,z\in\G_\la$$and hence defines  a
		Lorentzian bi-invariant metric $g$ on $G_\la$. The Levi-Civita connection of $g$ is given by $\na_XY=\frac12[X,Y]$ for any left invariant vector fields $X,Y$. Hence the left invariant vector field associated to $e_0$ is parallel and defines a Kundt group structure on $G_\la$.

		\end{exem}
\section{Kundt Spacetimes vs Kundt Groups}
\label{Section Kundt Spacetimes vs Kundt Groups}

In dimension 3, we have the following result.

\begin{theorem}  \label{main}
	Let $G$ be a Lie group of dimension 3  endowed with a left invariant Lorentz metric $g$. Assume 
	there exists a degenerate totally geodesic hypersurface $\Sigma$ in $(G, g)$ (not necessarily complete). \begin{enumerate}
		
		\item[$(i)$] Then, either  $(G, g)$ has   a constant sectional curvature, or 
		$(G, g)$ is a locally  Kundt Lie group.
		\item[$(ii)$]  Furthermore, if the isotropy group of 1 in the isometry group 
		$\mathsf{Iso}(G, g)$ is non-compact, then such a hypersurface $\Sigma$ exists.
	\end{enumerate}
	
\end{theorem}

To prove this theorem, we need the following lemma which also appears  in  \cite{Melnick, Scot, Zeg2}. 

\begin{Le} \label{3.geodesic.hypersurfaces} 
	Let $(G,g)$ be a Lorentz Lie group of dimension 3 .
	If there are three non-tangent  (i.e. having different tangent planes at 1) lightlike geodesic hypersurfaces through 1, then 
	$(G, g)$ has constant curvature.
	
\end{Le}

\begin{proof} 
	
	Assume there exist three different lightlike geodesic hypersurfaces $\Sigma_1$, $\Sigma_2$ and $\Sigma_3$, through 
	$1 \in G$.  Let $T_1, T_2, T_3$ be their tangent spaces and $V_1, V_2, V_3$,  non-vanishing vectors in their 
	orthogonals $T_1^\perp, T_2^\perp $ and $T_3^\perp$. 
	
	Since the $\Sigma_i$'s are geodesic, each 
	$T_i$ is  invariant under the Riemann curvature: if $u, v, w \in T_i$, then  $R(u, v)w \in T_i$.

	Let $e_3 $ be  a unit (spacelike)  vector generating $T_1 \cap T_2$, and consider $A_{e_3}: V \to R(e_3, V)e_3$.
	Then $ \langle A_{e_3}(V_1), e_3\rangle  = \langle R(e_3, V_1)e_3, e_3 \rangle = 0$, that is $A_{e_3}(V_1)$ is
	orthogonal to   $e_3$.  But it is also orthogonal to $V_1$, by curvature-invariance of 
	$V_1^\perp$. Therefore,   $A_{e_3}(V_1)$ is  collinear to $V_1$, say $A_{e_3}(V_1) = \lambda_{12} V_1$. 
	Similarly $A_{e_3}(V_2) = \lambda_{21} V_2$.

	Consider $\langle R(e_3, V_1)e_3, V_2 \rangle$ , 
	which   equals $\langle R(e_3, V_2)e_3, V_1 \rangle$. It also equals $\lambda_{12} \langle V_1, V_2 \rangle =
	\lambda_{21} \langle V_1, V_2 \rangle 
	$. Since $V_1$ and $V_2$ are both null, $ \langle V_1, V_2 \rangle \neq 0$ then
	$\lambda_{12} = \lambda_{21}$.  In conclusion $A_{e_3}$ is a homothety on $(e_3)^\perp$, of ratio, say $\lambda$ ($= \lambda_{12} = \lambda_{21}$).
	
	Since $A_{e_3}(e_3) = 0$, we conclude that $R(e_3, W) e_3 = \lambda \big(W \langle e_3,  e_3 \rangle - \langle W, e_3\rangle e_3 \big)$, for any $W$.

	In a similar way, one defines $e_2$ and $e_1$ unit vectors generating $T_1 \cap T_3$ and $T_2 \cap T_3$, respectively, and deduce 
	a similar formula for $R(e_i, W)e_i$, with the same $\lambda$.

	In fact, up to orientation, the $e_i$ are uniquely defined. More precisely,  in a  Lorentz linear 3-space $(E, \prs)$, there is up to isometry,  a unique system of three 
	null directions. Indeed, with the notation above, we can choose generators $V_1, V_2$ and $V_3$ of these directions such that 
	$\langle V_i, V_j \rangle = 1$, for any $i \neq j$, so the matrix coefficients $ \langle V_i, V_j \rangle_{ij}$ are fully given. This implies that any two such systems
	are related by an isometry. 
	
	Now, the system $\{e_1, e_2, e_3\}$ is given (up to orientation) by  $\{V_1, V_2, V_3\}$, so all scalar products 
	$\langle e_i, e_j \rangle$ are given, i.e. computable by means of $\prs$.
	
	It follows that  $R(e_i, e_j)e_i$ are given for any $i, j$.
	
	Consider now $X= R(e_i, e_j)e_k$, with $k \neq i, j$ and $i \neq j$. Then $\langle X, e_k \rangle = 0$, and 
	$\langle X, e_i \rangle  =  \langle R(e_i, e_j) e_i, e_k \rangle$ and, hence, computable. Similarly for $ \langle X, e_j \rangle$,
	and therefore $X$ is computable. 
	
	From all this, it follows that all the curvaturse $R(e_i, e_j)e_k$ are computable, exactly as in the case of a space of constant curvature $\lambda$.
\end{proof}

\paragraph{Proof of Theorem \ref{main}}
\begin{proof}
\begin{enumerate}\item[$(i)$] It follows form the lemma that if  $(G, g)$ does not have constant curvature, then, through any   point  pass exactly one or two 
	(germs of)  lightlike geodesic hypersurfaces. 
	
	For the sake of clarity, let us consider first the case where there exists exactly one germ of such hypersurfaces. More precisely, 
	there exists  a geodesic lightlike hypersurface $\Sigma$ containing 1. Uniqueness means that for any  $S$ a geodesic lightlike 
	hypersurface,  with $1 \in S$,  then $\Sigma \cap S$ is a neighbourhood of 1 in both $\Sigma$ and $S$.
	
	For any $x \in G$, the translated hypersurface  $x \Sigma$ is the unique geodesic lightlike hypersurface passing through $x$.
	
	Let us see that the tangent space of $\Sigma$ is left invariant: if $x \in \Sigma$, then $T_x \Sigma = x T_1 \Sigma$ (the last notation means  
	the left translation  by $x$ of $T_1 \Sigma$). Indeed, both $\Sigma$ and $x \Sigma$ are geodesic lightlike hypersurfaces containing $x$  hence they 
	coincide near $x$, and thus have same tangent space: $T_x \Sigma = T_x( x \Sigma) = x T_1 \Sigma$. 
	
	This left invariance of $T\Sigma$, means that $\Sigma$ is a ``local subgroup''. Its ``maximal extension''  will be a subgroup which is 
	a geodesic lightlike hypersurface. To be more formal, one defines a plane field $E$ on $G$, with $E(x)$ being the tangent space of the unique geodesic lightlike hypersurface through $x$. So,  $E(x) = T_x (x \Sigma)$ which implies $E$ is left invariant. Uniqueness implies $E$  is integrable: $\Sigma$ is a local leaf of $E$  trough $1$.
	The global leaf of $E$ is a subgroup.
	
	\bigskip

	Let us now consider the case where we have two geodesic lightlike hypersurfaces $\Sigma^1$ and $\Sigma^2$ through 1, which do not coincide 
	near 1. this is equivalent to $ E(1) = T_1 \Sigma^1 \neq T_1 \Sigma^2 = F(1)$. 
	
	Let $x \in \Sigma^1$. Then through $x$, we have three geodesic lightlike hypersurfaces: $\Sigma^1, x \Sigma^1 $ and $x \Sigma^2$. Therefore, two among them coincide locally.  Let us assume $x$ is close to 1 and deduce that $\Sigma^1 $ and $x \Sigma^2$ can not coincide near $x$.  For this, it is enough to show they have different tangent spaces $A= T_x \Sigma^1 \neq T_x (x \Sigma^2) = B$.  If they were equal, they would have the same  translation to 1, $x^{-1} A = x^{-1} B$. On one hand, $x^{-1} B = T_x \Sigma^2 = F(1) $, and on the other hand, $x^{-1} A= x^{-1} T_x \Sigma$ is close to $E(1) = T_1 \Sigma^1$, by continuity of the tangent space of  $\Sigma^1$ and 
	the fact that $x$ is close to 1. Since $E(1)$ and $F(1)$ are transversal, the same is true for $x^{-1} A $ and $  x^{-1} B$, for $x$ sufficiently close to 1, and in particular 
	they are not equal, hence    $A \neq B$.  From all of this, we infer that    $\Sigma^1$ and $x \Sigma^1$   coincide near $x$.

	As in the case of a unique geodesic lightlike hypersurface, one deduces that $\Sigma^1$ is a ``local group''. More precisely, one defines two 
	left invariant plane fields $E$ and $F$, extending $E(1)$ and $F(1)$ respectively. Our previous argument implies that $\Sigma^1$ is a local  leaf 
	at 1 of $E$, and so $E$ is integrable, and the same applies for $F$. The $E$ and $F$-leaves of 1 are therefore two 
	subgroups which are   geodesic lightlike hypersurfaces.

	\item[$(ii)$] Now, assume,  
	  $I$, the isotropy group of 1 in  the isometry group $\Iso(G, g)$  is non-compact. 
	Let $f_n \in I$ be a diverging sequence (it has non-convergent sub-sequence) and consider their  graphs $F_n= Gr(f_n) \subset G \times G$. Endow 
	$G \times G$ with the metric $g\oplus (-g)$, then the $F_n$'s are isotropic and totally geodesic.  Then we will  consider a limit $L$ of a subsequence of the $F_n$. 
	To give a formal meaning of this, consider a small convex neighbourhood $C$ of $(1, 1)$ in $G \times G$. This means any two points of $C$ can be joined by a unique
	geodesic segment contained in $C$. Consider $F_n \cap C$ and note $F_n^0$ the connected component of $(1, 1)$ in $F_n \cap C$. Now, one can give sense to 
	convergence of $F_n$, exactly as in the situation of affine subspaces in an affine flat space. More precisely, $F_n^0$ converge to $L$, if   the tangent spaces 
	$T_{(1,1)}F^0_n$ converge to $T_{(1,1)} L$. 
	
	Such a limit $L$ is a geodesic isotropic submanifold in $G \times G$ of dimension  equal to $\dim G$, but it is no longer a graph of some map $f: G \to G$, since otherwise, 
	$f_n$ will converge to $f$. Thus $L$ intersects non-trivially the vertical $\{1\} \times G$, and  hence projects onto a  degenerate  geodesic submanifold $\Sigma$ in $G \times \{1\}$.
	The intersection $L \cap (\{1\} \times G)$ has dimension one, since it is isotropic,  therefore $\Sigma$ is a hypersurface.

	\end{enumerate}

	\end{proof}

\subsection{Comments on the constant curvature case}

It is natural to ask what happens if $(G, g)$ has constant sectional curvature, say $c$?  Let us give here some examples and hints, details will appear elsewhere. The proof of  Theorem \ref{main} does not apply since there are infinitely many ``germs'' of lightlike geodesic hypersurfaces trough 1: any lightlike hyperplane in $T_1G$ is tangent to a   lightlike geodesic hypersurface exactly as in the 
universal Lorentz space
$\tilde{M}(c)$ of constant curvature $c$.  

\bigskip

$\bullet$ It turns out that if the isotropy group $I$ is non-compact and has dimension 1 or 2, then $(G, g)$ is a Kundt group. 
Indeed,  the proof of  Theorem \ref{main} can be  adapted  well to this situation   where $\dim I = 1$ or  $2$. Consider for this  the derivative action of $I$ on 
$\g = T_1G$. It  preserves exactly one or two hyperplanes which are in fact lightlike. Indeed,  let $L$ be a closed  connected non-compact subgroup
of ${\sf O }(1, 2) $. If $\dim L= 1$, then it is either a hyperbolic one parameter group, in this case it preserves exactly two lightlike hyperplanes, or 
it is a unipotent one parameter group and in this  case it preserves exactly one lightlike hyperplane. In case $\dim L= 2$, it is conjugate to the triangular subgroup of $\SL(2, \R)$ and preserves exactly one lightlike hyperplane. All these claims can  be confirmed by a direct check-in.  In summary, if $\dim I = 1, 2$, the argument in the proof 
of  Theorem \ref{main} can be adapted and yields a left invariant plane field $E$ and, thus,    a lightlike  geodesic subgroup $H$.

Let us give the following example with $c = 0$.  Consider on $\R^3$, the Lorentz metric $dx^2 + dydz$. The plane $E = \{ z=0\}$ is lightlike.
Its linear stabilizer  is a subgroup $S$ of dimension   2 in 
$\SO(1, 2)$ and, hence, its stabilizer in the full Poincar\'e group  $\SO(1, 2) \ltimes \R^3$ is $S \ltimes E$. It contains in particular the 3-dimensional
(non-unimodular)  group $G$ of elements $(t, a, b) \in \R \times E$ acting by 
$(x, y, z) \to (x + a, e^ty + b, e^{-t}z)$.   This action is   free  and  transitive on the upper half space $\{z >0\}$ and, hence, $G$ inherits a left invariant (non-complete) flat metric. It is Kundt, since its (abelian) subgroup $E$ has lightlike geodesic orbits.  Observe here that the full isometry group of this 
left invariant metric is $S \ltimes E$. In particular the  isotropy group has dimension one.

\bigskip

$\bullet$  
Assume now that the isotropy group has  dimension $3$. Thus $ \dim \Iso(G, g) = 6$, and $(G, g)$ is locally isometric to the universal space
$\tilde{M}(c)$ of constant curvature $c$.   However, a subgroup of dimension 3 in ${\sf O }(1, 2)$ contains at least its identity component, and also, a subgroup of dimension 6  in $\Iso (\tilde{M}(c))$ contains at least its identity component $\Iso^0 (\tilde{M}(c))$. Therefore, as a homogeneous space $(G, g)$ is globally 
isometric to $ \tilde{M}(c) = \Iso^0 (\tilde{M}(c)) / {\sf O }^0(1, 2)$. In other words, $G$ acts transitively and freely on $\tilde{M}(c)$ (or equivalently, $(G, g)$ has constant curvature $c$ and is complete). 

Let us give the example of the Euclide group $\Euc_2$.  Its universal cover $\widetilde{\Euc_2}$ acts simply transitively isometrically 
on $(\R^3, dx^2 + dy^2 - dz^2)$ by: 
$((x, y), z)   \to \big( (R_t(x, y) + (a, b)), z+ t \big)$, where $R_t$ is the rotation of angle $t$ and $(t, a, b) \in \widetilde{\Euc_2}$. Its unique 
2-dimensional subgroup is $\R^2$, which  acts by translations $(x, y, z) \to (x+a, y+b, z)$. It has spacelike geodesic orbits. Therefore
$\widetilde{\Euc_2}$ is a flat complete Lorentz group that is not  a Kundt group.

We believe this is the unique complete Lorentz group of constant curvature which is not a Kundt group?

\bigskip

$\bullet$ Finally, there are examples of flat groups $(G, g)$ with isotropy group $I$ of dimension 0, which are not Kundt groups. To see an example, consider as above the metric $dx^2 + dydz$.   Let $S \subset \SO(1, 2)$ be  the stabilizer of the isotropic direction 
$\R \frac{\partial }{\partial y}$ and $T$ the subgroup of translations in this direction. Take $G = S \ltimes T$. It has an open orbit on which it acts freely which allows one to endow it with a flat (non-complete) metric. One can show it is not a Kundt group.

\section{Classification of three dimensional unimodular simply-connected Kundt Lie groups}\label{section3}

 In this section, we give a complete classification of Kundt Lie group structures on three dimensional unimodular Lie groups.
 According to Proposition \ref{kundt},  the classification of Kundt Lie group structures on a simply connected Lie group $G$ is equivalent to the classification of Kundt pairs $(\h,\prs)$ on its Lie algebra $\G$. Two Kundt pairs $(\h_1,\prs_1)$ and $(\h_2,\prs_2)$ are called equivalent if there exists an automorphism of Lie algebra $\phi:\G\too\G$ such that $\phi(\h_1)=\h_2$ and $\phi^*(\prs_2)=\prs_1$.

In dimension 3, we have the following useful characterization of Kundt pairs.
\begin{pr} \label{dim}  Let $(\g,\prs)$ be a  Lorentzian Lie algebra and let $\h$ be a  codimension one subalgebra. Then:
	\begin{enumerate}\item[$(i)$] If $\h$ is abelian then $(\h,\prs)$ is a Kundt pair if and only if $\h$ is degenerate and if $e$ is a generator of $\h^\perp$ then $\ad_e(\G)\subset\h$.
		\item[$(ii)$] If  $\h$ is non-abelian then $(\h,\prs)$ is a Kundt pair if and only if $\h^\perp=[\h,\h]$ and if $e$ is a generator of $\h^\perp$ then $\ad_e(\G)\subset\h$.
		
	\end{enumerate}

\end{pr}

\begin{proof} Let $e$ a generator of $\h^\perp$. Then $(\h,\prs)$ is a Kundt pair if and only if $e\in\h^\perp$, $\langle e,e\rangle =0$, $e\bullet e=0$ and for any $u,v\in\h$,
	\begin{equation}\label{eq} 0=2\langle u\bullet v,e\rangle = \langle [e,u],v\rangle +\langle [e,v],u\rangle. \end{equation}
	Note first that $e\bullet e=0$ and $\langle e,e\rangle=0$ if and only if, for any $u\in\G$,
	\[ 0=\langle e\bullet e,x\rangle =\langle [u,e],e\rangle \]which is equivalent to 
	$\ad_e(\g)\subset\h$.
	
		If $\h$ is abelian then \eqref{eq} holds trivially. 
		
		Suppose now that $\dim\G=3$ and  $\h$ is not abelian. Then there exists a basis $(u,v)$ of $\h$ such $\langle u,v\rangle =0$ and  $[u,v]=u$. Put $e=au+bv$. Then from the relation above, we get
		\[ 0=-b\langle u,v\rangle +a\langle u,u\rangle=a\langle u,u\rangle\esp 0=\langle [e,u],u\rangle=b\langle u,u\rangle. \]
		This implies that $\langle u,u\rangle =0$ and hence $\h^\perp=[\h,\h]$. The converse is obviously true.

\end{proof}

	There are five simply connected three dimensional unimodular non abelian Lie groups: \begin{enumerate}
		
		\item The nilpotent Lie group $\nil$ known as Heisenberg group whose Lie algebra will be denoted by $\mathfrak{n}$. We have
		\[ \nil=\left\{\left(\begin{matrix} 1&x&z\\0&1&y\\0&0&1\end{matrix}\right),x,y,z\in\R   \right\}
		\esp \mathfrak{n}=\left\{\left(\begin{matrix} 0&x&z\\0&0&y\\0&0&0\end{matrix}\right),x,y,z\in\R   \right\}. \]
		The Lie algebra $\n$ has a basis $\B_0=(X_1,X_2,X_3)$ where
		\[ X_1=\left(\begin{matrix} 0&1&0\\0&0&0\\0&0&0\end{matrix}\right), X_2=\left(\begin{matrix} 0&0&0\\0&0&1\\0&0&0\end{matrix}\right)\esp X_3=\left(\begin{matrix} 0&0&1\\0&0&0\\0&0&0\end{matrix}\right) \]
		where   the non-vanishing Lie bracket is $[X_1,X_2]=X_3$. 
		\item $\mathrm{SU}(2)=\left\{\left(\begin{matrix} a+bi&-c+di\\c+di&a-bi\end{matrix}\right),a^2+b^2+c^2+d^2=1   \right\}
		\esp \mathfrak{su}(2)=\left\{\left(\begin{matrix} iz&y+ix\\-y+xi&-zi\end{matrix}\right),x,y,z\in\R   \right\}.$ The Lie algebra $\mathfrak{su}(2)$ has a basis $\B_0=(X_1,X_2,X_3)$
		\[ X_1=\frac12\left(\begin{matrix}0&i\\i&0\end{matrix} \right),\; X_2=\frac12\left(\begin{matrix}0&1\\-1&0\end{matrix} \right)\esp X_3=\frac12\left(\begin{matrix}-i&0\\0&i\end{matrix} \right) \]
	 where   the non-vanishing Lie brackets are
		\begin{equation*} \label{eqsu2}[X_1,X_2]=X_3,\;[X_2,X_3]=X_1\esp [X_3,X_1]=X_2. \end{equation*} 
		
		\item The universal covering group $\wi{\mathrm{PSL}}(2,\R)$ of $\mathrm{SL}(2,\R)$ whose Lie algebra  is $\mathrm{sl}(2,\R)$. 
		The Lie algebra $\mathrm{sl}(2,\R)$ has a basis $\B_0=(e,f,h)$ where
		$$e=\left(\begin{matrix}
		0&1\\0&0
		\end{matrix}  \right), \ f=\left(\begin{matrix}
		0&0\\1&0
		\end{matrix}  \right), \ h=\left(\begin{matrix}
		1&0\\0&-1
		\end{matrix}  \right).$$ where
		the non-vanishing Lie brackets are
		\begin{equation*} \label{eqsl}[e,f]=h,\;[h,e]=2e\esp [h,f]=-2f. \end{equation*} 
		\item The solvable Lie group $\sol=\left\{\left(\begin{matrix} e^x&0&y\\0&e^{-x}&z\\0&0&1\end{matrix}\right),x,y,z\in\R   \right\}$ whose Lie algebra is $\mathfrak{sol}=\left\{\left(\begin{matrix} x&0&y\\0&-x&z\\0&0&0\end{matrix}\right),x,y,z\in\R   \right\}$. The Lie algebra $\mathfrak{sol}$ has a basis $\B_0=(X_1,X_2,X_3)$ where
		\[ X_1=\left(\begin{matrix} 1&0&0\\0&-1&0\\0&0&0\end{matrix}\right),\;
		X_2=\left(\begin{matrix} 0&0&1\\0&0&0\\0&0&0\end{matrix}\right),\;
		X_3=\left(\begin{matrix} 0&0&0\\0&0&1\\0&0&0\end{matrix}\right)\esp  \]  where the non-vanishing Lie brackets are
		\begin{equation*} \label{eqsol} [X_1,X_2]=X_2\esp [X_1,X_3]=-X_3. \end{equation*}
		\item The universal covering group $\wi{\mathrm{E}_0}(2)$ of the Lie group 
		\[ \mathrm{E}_0(2)=\left\{\left(\begin{matrix} \cos(\theta)&\sin(\theta)&x\\-\sin(\theta)&\cos(\theta)&y\\0&0&1\end{matrix}\right),\theta,x,y\in\R   \right\}. \]Its Lie algebra is
		\[ \mathrm{e}_0(2)=\left\{\left(\begin{matrix} 0&\theta&x\\-\theta&0&y\\0&0&0\end{matrix}\right),\theta,y,z\in\R   \right\}.  \]The Lie algebra $\mathrm{e}_0(2)$ has a basis $\B_0=(X_1,X_2,X_3)$ where
		\[ X_1=\left(\begin{matrix} 0&-1&0\\1&0&0\\0&0&0\end{matrix}\right),\;
		X_2=\left(\begin{matrix} 0&0&1\\0&0&0\\0&0&0\end{matrix}\right),\;\esp
		X_3=\left(\begin{matrix} 0&0&0\\0&0&1\\0&0&0\end{matrix}\right) \]
		where  the non-vanishing Lie brackets are
		\begin{equation*} \label{eqe2} [X_1,X_2]=X_3\esp [X_1,X_3]=-X_2. \end{equation*}

	\end{enumerate}
	
	Let us find the 2-dimensional subalgebras of the 3-dimensional unimodular Lie algebras.
	
	\begin{pr}\label{class}
		\begin{enumerate}\item Let $\h$ be a 2-dimensional subalgebra of $\mathfrak{n}$. Then $\h=\mathrm{span}\{X_3,a X_1+b X_2 \}$, $(a,b)\not=(0,0)$.
			
			\item $\mathfrak{su}(2)$ has no subalgebra of dimension 2.

			\item Let $\h$ be a 2-dimensional subalgebra of $\mathfrak{sol}$ then either $\h=\mathrm{span}\{X_2,X_3\}$, $\h=\mathrm{span}\{X_2,X_1+aX_3\}$ or
			$\h=\mathrm{span}\{X_3,X_1+aX_2\}$ ($a\in\R)$
			
			\item Let $\h$ be a 2-dimensional subalgebra of $e_0(2)$ then $\h=\mathrm{span}\{X_2,X_3\}$.
			
			\item Let $\h$ be 2-dimensional subalgebra of $\mathrm{sl}(2,\R)$. Then there exists an automorphism of $\mathrm{sl}(2,\R)$ which sends $\h$ to 
			$\mathrm{span}\{h,e\}$.  
			
		\end{enumerate}

	\end{pr}
	
	\begin{proof} \begin{enumerate}\item A 2-dimensional subalgebra $\h$ of $\mathfrak{n}$ must be abelian and contains the center. So $\h=\mathrm{span}\{X_3,a X_1+b X_2 \}$ and $(a,b)\not=(0,0)$.
		\item It is a consequence of Lemma \ref{compact}.
		\item Denote by $\h_0=\mathrm{span}\{X_2,X_3 \}$. If $\h$ is abelian and $\h\not=\h_0$ then $\h =\mathrm{span}\{ X_1+U,V \}$ where $U,V\in\h_0$  hence $[X_1,V]=0$ which is impossible. So if $\h$ is abelian then $\h=\h_0$. 
		
		 If $\h$ is not abelian then  $\h\not=\h_0$. Thus $\h =\mathrm{span}\{ X_1+U,V \}$ where $U,V=aX_2+bX_3\in\h_0$ and $[\h,\h]=\R V$. Now
		 \[ [X_1+U,V] = aX_2-b X_3. \]	So the vectors $aX_2+bX_3,aX_2-bX_3$ must be linearly dependent, hence $ab=0$. Which complete the proof.
		 \item We can use the same argument as above and get $a^2+b^2=0$.
		 \item Note first that $[e,f]=h, \ [h,e]=2e, \ [h,f]=-2f$. Let $\h$ be a 2-dimensional subalgebra of $\mathrm{sl}(2,\R)$. Then there exists a basis $(u,v)$ of $\h$ such that $[u,v]=2v$. The endomorphism $\ad_u$ is skew-symmetric with respect to the Killing form hence $\tr(\ad_u)=0$. It has $2$ and $0$ as eigenvalues so the third eigenvalue is $-2$. So there exists $w\in\mathrm{sl}(2,\R)$ such that $[u,w]=-2w$. Now
		 \[ [u,[v,w]]=2[v,w]-2[v,w]=0 \]  hence $[v,w]=\al u$. By replacing $w$ by $\frac1\al w$ we get that the automorphism $\phi$ which sends $(u,v,w)$ to $(h,e,f)$ sends $\h$ to $\mathrm{span}\{h,e\}$ which completes the proof.

			\end{enumerate}
		
	\end{proof}
	
	\begin{theorem}
		Let $(\h,\prs)$ be a Kundt pair of $\mathfrak{n}$. Then $(\h,\langle,\rangle)$ is equivalent to $(\h_0,\langle,\rangle_0)$ where either:
		\begin{enumerate}\item $\langle,\rangle_0=\left[ \begin {array}{ccc} 1 &0&0\\ \noalign{\medskip}0&-1&0
			\\ \noalign{\medskip}0&0&\mu\end {array} \right], \mu >0$  and $\h_0=\mathrm{span}\{X_1\pm X_2,X_3\}$.
			
			\item $ \langle,\rangle_0=\left[\begin{array}{ccc}1&0&0\\0&0&1\\0&1&0  \end{array}\right]$ and $\h_0=\mathrm{span}\{X_1,X_3\}$.
		\end{enumerate}
	\end{theorem}

	\begin{proof} Let $(\h,\prs)$ be a Kundt pair on $\mathfrak{n}$. According to \cite[Theorem 3.1]{chakkar}, there exists an automorphism $\phi$ of $\mathfrak{n}$ such that the matrix of $(\phi^{-1})^*(\prs)$ in the basis $(X_1,X_2,X_3)$ has one of the following forms:
	\[ \n_1=\left[ \begin {array}{ccc} 1 &0&0\\ \noalign{\medskip}0&-1&0
	\\ \noalign{\medskip}0&0&\mu\end {array} \right],\;
	\n_2=\left[ \begin {array}{ccc} 1 &0&0\\ \noalign{\medskip}0&1&0
	\\ \noalign{\medskip}0&0&-\mu\end {array} \right]  \ \text{or}  \ \n_3=\left[ \begin {array}{ccc} 1 &0&0\\ \noalign{\medskip}0&0&1
	\\ \noalign{\medskip}0&1&0\end {array} \right],\quad \mu>0. \]	 
		By virtue of Proposition \ref{class}, $\phi(\h)=\h_0=\mathrm{span}\{ X_3,aX_1+bX_2  \}$ with $(a,b)\not=(0,0)$. According to Proposition \ref{dim}, $(\phi(\h),(\phi^{-1})^*(\prs))$ is a Kundt pair if and only if $\h_0$ is degenerate and $\ad_{e}(\n)\subset\h_0$ where $e$ is a generator of $\h_0^\perp$. Since $[\n,\n]\subset\h_0$ then the last condition holds.
		
		Now, $\h_0$ cannot be $\n_2$-degenerate and it is $\n_3$-degenerate if and only if $b=0$. Finally, $\h_0$ is $\n_1$-degenerate if and only if $a^2-b^2=0$ which completes the proof.
	\end{proof}

\begin{theorem} Let $(\h,\prs)$ be a Kundt pair of $\mathfrak{sol}$. Then $(\h,\langle,\rangle)$ is equivalent to $(\h_0,\prs_0)$ where either:
	\begin{enumerate}\item  $\prs_0=\left(\begin{matrix}
		\la&0&0\\0&0&-1\\
		0&-1&0&
		\end{matrix}  \right), \la>0$ and $\h_0=\mathrm{span}\{X_2,X_1\}$ or $\h_0=\mathrm{span}\{X_3,X_1\}$,
		
		\item $\prs_0=\left(\begin{matrix}
		\la^2&0&0\\0&\la&1\\
		0&1&0
		\end{matrix}  \right)$ and $\h_0=\mathrm{span}\{X_3,X_1\}$,
	\item $\prs_0=\left(\begin{matrix}
	0&0&-\frac{2}{b}\\0&1&1\\
	-\frac{2}{b}&1&1
	\end{matrix}  \right),\quad b>0$ and $\h_0=	\mathrm{span}\{X_2,X_3\}$.

	\item  $ \langle,\rangle_0=\left[\begin{array}{ccc}0&0&1\\0&1&0\\1&0&0  \end{array}\right]$ and $\h_0=\mathrm{span}\{X_2,X_3\}$.
	
	\end{enumerate}
\end{theorem}

\begin{proof} Let $(\h,\prs)$ be a Kundt pair on $\mathfrak{sol}$. Then according to \cite[Theorem 3.4]{chakkar}, there exists an automorphism $\phi$ of $\mathfrak{sol}$ such that the matrix of $(\phi^{-1})^*(\prs)$ in the basis $(X_1,X_2,X_3)$ has one of the following forms:
	$$
	\begin{cases}
		\mathrm{sol}_1=\left(\begin{matrix}
			\frac4{u^2-v^2}&0&0\\0&1&\frac{u}{v}\\
			0&\frac{u}{v}&1
		\end{matrix}  \right), v>0, u<v,\;\mathrm{sol}_2=\left(\begin{matrix}
		\frac4{v^2-u^2}&0&0\\0&\frac{u}{v}&-1\\
		0&-1&\frac{u}{v}&
		\end{matrix}  \right), v>0, u<v,\; \mathrm{sol}_3=\left(\begin{matrix}
		\frac1{u+v}&0&0\\0&-\frac{v}{u}&1\\
		0&1&1
		\end{matrix}  \right),\;u>0, v>0,\\
			\mathrm{sol}_4=\left(\begin{matrix}
			\frac1{u}&0&0\\0&-1&0\\
			0&0&1
			\end{matrix}  \right),\;u>0,\;
			
			\mathrm{sol}_5=\left(\begin{matrix}
			0&0&-\frac{2}{b}\\0&1&1\\
			-\frac{2}{b}&1&1
			\end{matrix}  \right),\quad b>0,\;
			
			\mathrm{sol}_6=\left(\begin{matrix}
			\la^2&0&0\\0&\la&1\\
			0&1&0
			\end{matrix}  \right),\quad \la\not=0,\;\mathrm{sol}_7=\left(\begin{matrix}
			0&0&1\\0&1&0\\1&0&0
			\end{matrix}  \right).
			\end{cases}
			$$

	By virtue of Proposition \ref{class}, $\phi(\h)=\h_0$ where either $\h_0=\mathrm{span}\{X_2,X_3\}$, $\h_0=\mathrm{span}\{X_2,X_1+aX_3\}$ or
	$\h_0=\mathrm{span}\{X_3,X_1+aX_2\}$ ($a\in\R)$. 
	
	If $\h_0=\mathrm{span}\{X_2,X_3\}$ then it is abelian and it is $\mathrm{sol}_5$-degenerate and $\mathrm{sol}_7$-degenerate. For $\mathrm{sol}_5$, $\h_0^\perp=\R(X_2-X_3)$ and we have $\ad_{X_2-X_3}(\mathfrak{sol})\subset \h_0$. We have the same situation for $\mathrm{sol}_7$. Thus $(\h_0,\mathrm{sol}_5)$ and 
	$(\h_0,\mathrm{sol}_7)$ are Kundt pairs.
	
	If $\h_0=\mathrm{span}\{X_2,X_1+aX_3\}$ then $[\h_0,\h_0]=\R X_2$. We have obviously, $\ad_{X_2}(\mathfrak{sol})\subset\h_0$ and, according to Proposition \ref{dim}, $(\h_0,\prs)$ is a Kundt pair if and only if $\langle X_2,X_2\rangle=\langle X_2,X_1+aX_3\rangle =0$. This is possible if and only if $\prs=\mathrm{sol}_2$ with $u=0$ and $a=0$.
	
	If $\h_0=\mathrm{span}\{X_3,X_1+aX_2\}$ then $[\h_0,\h_0]=\R X_3$. We have obviously, $\ad_{X_3}(\mathfrak{sol})\subset\h_0$ and, according to Proposition \ref{dim}, $(\h_0,\prs)$ is a Kundt pair if and only if $\langle X_3,X_3\rangle=\langle X_3,X_1+aX_2\rangle =0$. This is possible if and only if $\prs=\mathrm{sol}_2$, $u=0$ and $a=0$ or $\prs=\mathrm{sol}_6$ and $a=0$.
	\end{proof}

\begin{theorem} Let $(\h,\prs)$ be a Kundt structure on $e_0(2)$. Then $(\h,\prs)$ is equivalent to $(\h_0,\prs_0)$ where\\
	$\prs_0=\left[ \begin {array}{ccc} 0&1&0
	\\ \noalign{\medskip}1&0&0\\ \noalign{\medskip}0&0&\mu\end {array}
	\right], \mu >0$ and $\h_0=\mathrm{span}\{X_2,X_3\}$.
\end{theorem}
\begin{proof}Let $(\h,\prs)$ be a Kundt pair on $e_0(2)$. Then according to \cite[Theorem 3.5]{chakkar}, there exists an automorphism $\phi$ of $e_0(2)$ such that the matrix of $(\phi^{-1})^*(\prs)$ in the basis $(X_1,X_2,X_3)$ has one of the following forms:
	\[ \prs_1=\left(\begin{matrix}
	0&1&0\\1&u&0\\
	0&0&v
	\end{matrix}  \right),\;u>0,v>0,\;\prs_2=\left(\begin{matrix}
	0&1&0\\1&u&0\\
	0&0&v
	\end{matrix}  \right),\;u>0,v>0,\;\prs_3=\left(\begin{matrix}
	0&1&0\\1&0&0\\
	0&0&u
	\end{matrix}  \right), \quad u>0. \]
		 
	By virtue of Proposition \ref{class}, $\phi(\h)=\h_0=\mathrm{span}\{ X_2,X_3  \}$. According to Proposition \ref{dim}, $(\phi(\h),(\phi^{-1})^*(\prs))$ is a Kundt pair if and only if $\h_0$ is degenerate and $\ad_{e}(\n)\subset\h_0$ where $e$ is a generator of $\h_0^\perp$. 
	Now, $\h_0$ cannot be $\prs_1$-degenerate neither $\prs_2$-degenerate. Finally, $\h_0$ is $\prs_3$-degenerate,  $\h_0^\perp=\R X_2$ and $\ad_{X_2}(e_0(2))\subset\h_0$.
		\end{proof}

\begin{theorem} Let $(\h,\prs)$ be a Kundt pair on $\mathrm{sl}(2,\R)$. Then there exists an automorphism $\phi$ of $\mathrm{sl}(2,\R)$ such that $\phi(\h)=\mathrm{span}\{e,h\}$ and the matrix of $\phi^*(\prs)$ in the  basis $(e,f,h)$ has one of the following forms:
	
$$\left[ \begin {array}{ccc} 0&4\,\alpha&0\\ \noalign{\medskip}4\,
	\alpha&0&0\\ \noalign{\medskip}0&0&8\,\beta\end {array} \right],\;
	\left[ \begin {array}{ccc} 0&4\,\alpha&0\\ \noalign{\medskip}4\,
	\alpha&1&0\\ \noalign{\medskip}0&0&8\,\beta\end {array} \right]\quad\mbox{or}\quad
	 \left[ \begin {array}{ccc} 0&4\,\alpha&0\\ \noalign{\medskip}4\,
	 \alpha&0&2\,\sqrt {2}\\ \noalign{\medskip}0&2\,\sqrt {2}&8\,\alpha
	 \end {array} \right], \ \beta>0, \ \alpha \in \mathbb{R}^*$$

	\end{theorem}

	\begin{proof} Let $(\h,\prs)$ be a Kundt pair on $\mathrm{sl}(2,\R)$. According to Proposition \ref{class}, we can suppose that $\h=\mathrm{span}\{e,h\}$. A direct computation using the software Maple shows that the automorphisms of $\mathrm{sl}(2,\R)$ leaving $\h$ invariant are of the form
		\[ T=\left[ \begin {array}{ccc} a&-ab^2&-2\,ab\\ \noalign{\medskip}0&
			a^{-1}&0\\ \noalign{\medskip}0&b&1
		\end {array} \right],\; a,b\in\R. 
		 \]Denote by  $B$ the Killing form of $\mathrm{sl}(2,\R)$. Its matrix in the basis $(e,f,h)$ is given by
		 \[ M=\left[ \begin {array}{ccc} 0&4&0\\ \noalign{\medskip}4&0&0
		 \\ \noalign{\medskip}0&0&8\end {array} \right]. 
		 \] We consider the isomorphism $A$, symmetric with respect to $B$, 	and given by $B(Au,v)=\langle u,v\rangle$ for any $u,v\in \mathrm{sl}(2,\R)$. According to Proposition \ref{dim},  the pair $(\h,\prs)$ is Kundt if and only if $\langle e,e\rangle=\langle e,h\rangle =0$ and $\ad_e(\mathrm{sl}(2,\R))\subset\h$. This is equivalent to $Ae$ is a generator of the orthogonal of $\h$ with respect to $B$ which is equivalent to the existence of $\al\not=0$ such that $Ae=\al e$. The normal form of isomorphisms which are  symmetric  with respect to a Lorentzian scalar product are known. We give here the normal form of those having an isotropic eigenvector.  According to \cite{Oneil}, there exits a basis $\B=(f_1,f_2,f_3)$ of $\mathrm{sl}(2,\R)$ such that:
		 \begin{enumerate}\item[$(i)$] $Mat(A,\B)=\left[ \begin {array}{ccc} \beta&0&0\\ \noalign{\medskip}0&\alpha&0
		 	\\ \noalign{\medskip}0&0&\alpha\end {array} \right]
		 	$ and $Mat(B,\B)=\left[ \begin {array}{ccc} 1&0&0\\ \noalign{\medskip}0&0&1
		 	\\ \noalign{\medskip}0&1&0\end {array} \right]$,
		 	\item[$(ii)$] $Mat(A,\B)=\left[ \begin {array}{ccc} \beta&0&0\\ \noalign{\medskip}0&\alpha&1
		 	\\ \noalign{\medskip}0&0&\alpha\end {array} \right]$ and $Mat(B,\B)=\left[ \begin {array}{ccc} 1&0&0\\ \noalign{\medskip}0&0&1
		 	\\ \noalign{\medskip}0&1&0\end {array} \right],$
		 	\item[$(iii)$]  $Mat(A,\B)=\left[ \begin {array}{ccc} \alpha&1&0\\ \noalign{\medskip}0&\alpha&1
		 	\\ \noalign{\medskip}0&0&\alpha\end {array} \right]$ and  
		 	$Mat(B,\B)=\left[ \begin {array}{ccc} 0&0&1\\ \noalign{\medskip}0&1&0
		 	\\ \noalign{\medskip}1&0&0\end {array} \right]$.

		 	\end{enumerate}
		 Let $P$ be the passage matrix from $\B_0=(e,f,h)$ to $\B$.

		$\bullet$  The case $(i)$. The relation $A(e)=\al e$ implies that we can choose $f_2=e$ and the relation $P^tMP=Mat(B,\B)$ gives that $P=\left[ \begin {array}{ccc} -\sqrt {2}&1&-1\\ \noalign{\medskip}0&0&1/
		 4\\ \noalign{\medskip}1/4\,\sqrt {2}&0&1/2\end {array} \right]$. The matrix of $\prs$ in the basis $B_0$ is given by $Mat(A,\B_0)^tM$ and $Mat(A,\B^0)=PMat(A,\B)P^{-1}$. So we get
		 \[ Mat(\prs,\B_0)=\left[ \begin {array}{ccc} 0&4\,\alpha&0\\ \noalign{\medskip}4\,
		 \alpha&32\,\beta-32\,\alpha&-16\,\beta+16\,\alpha\\ \noalign{\medskip}0
		 &-16\,\beta+16\,\alpha&8\,\beta\end {array} \right].
		  \] Now the automorphism
		  \[ T_1=\left[ \begin {array}{ccc} 1&-4&-4\\ \noalign{\medskip}0&1&0
		  \\ \noalign{\medskip}0&2&1\end {array} \right]
		   \]satisfies
		   \[ T_1^tMat(\prs,\B_0)T_1=\left[ \begin {array}{ccc} 0&4\,\alpha&0\\ \noalign{\medskip}4\,
		   \alpha&0&0\\ \noalign{\medskip}0&0&8\,\beta\end {array} \right].
		    \]
		 $\bullet$  The case $(ii)$. We have
		 \[ P=\left[ \begin {array}{ccc} -\sqrt {2}&1&-1\\ \noalign{\medskip}0&0&1/
		 4\\ \noalign{\medskip}1/4\,\sqrt {2}&0&1/2\end {array} \right]\esp 
		 T_2=\left[ \begin {array}{ccc} 4&-1&-4\\ \noalign{\medskip}0&1/4&0
		 \\ \noalign{\medskip}0&1/2&1\end {array} \right] 
		  \]
		\[ T_2^tMat(\prs,\B_0)T_2=\left[ \begin {array}{ccc} 0&4\,\alpha&0\\ \noalign{\medskip}4\,
		\alpha&1&0\\ \noalign{\medskip}0&0&8\,\beta\end {array} \right].
		\]
		
		$\bullet$  The case $(iii)$. We have
		\[ P=\left[ \begin {array}{ccc} 1&-\sqrt {2}&-1\\ \noalign{\medskip}0&0&1/
		4\\ \noalign{\medskip}0&1/4\,\sqrt {2}&1/2\end {array} \right] 
		\esp T_3=\left[ \begin {array}{ccc} 4&-1&-4\\ \noalign{\medskip}0&1/4&0
		\\ \noalign{\medskip}0&1/2&1\end {array} \right] 
		\]
		  \[ T_3^tMat(\prs,\B_0)T_3=\left[ \begin {array}{ccc} 0&4\,\alpha&0\\ \noalign{\medskip}4\,
		  \alpha&0&2\,\sqrt {2}\\ \noalign{\medskip}0&2\,\sqrt {2}&8\,\alpha
		  \end {array} \right].		  
		  \]

		\end{proof}

\subsection{Kundt vs Locally Kundt Lie Groups}
\label{Kundt vs Locally Kundt Groups}

In fact, it turns out  from the previous proofs, we have shown that any 3-dimensional unimodular locally Kundt Lie group is  in fact a Kundt Lie group. 
This result is not  true in  general as the following example shows.

\begin{exem} Consider $\R^4$ endowed with the Lie algebra structure where the only non vanishing Lie bracket is given by $[e_1,e_2]=e_2$ and the Lorentzian scalar product given by
	\[ \prs=\left(\begin{array}{cccc}0&1&0&0\\1&0&0&0\\0&0&1&0\\0&0&0&1   \end{array}   \right). \]
	The Lie subalgebra $\h=\mathrm{span}\{e_1,e_3,e_4 \}$ is abelian and satisfies $\h^\perp=\R e_1$ and hence, according to Proposition \ref{kundt}, defines a local Kundt Lie group structure on the corresponding simply connected Lie group. However, $\ad_{e_1}(\R^4)\nsubset\h$ and hence, according to Proposition \ref{dim} this structure is not global.

	\end{exem}

We think however, it is worthwhile to investigate the natural question: is a 
  locally Kundt Lie group, without being  a Kundt group,  still   a (globally) Kundt spacetime?

\end{document}